\documentclass[11pt, a4paper, twoside,leqno]{amsart}

\usepackage{amsmath}
\usepackage{amssymb}
\usepackage{microtype}
\usepackage[shortlabels]{enumitem}
\usepackage[british]{babel}
\input diagxy
\usepackage[colorlinks=true]{hyperref}
\hypersetup{allcolors=[rgb]{0.1,0.1,0.4}}

\renewcommand{\phi}{\varphi}
\renewcommand{\epsilon}{\varepsilon}

\newcommand{\E}{{\mathcal E}}

\newcommand{\M}{{\mathcal M}}

\numberwithin{equation}{section}

\theoremstyle{plain}
\newtheorem{Thm}{Theorem}
\newtheorem{Prop}[Thm]{Proposition}
\newtheorem{Cor}[Thm]{Corollary}
\newtheorem{Lemma}[Thm]{Lemma}

\theoremstyle{definition}
\newtheorem{Defn}[Thm]{Definition}

\newtheorem{Ex}[Thm]{Example}

\newtheorem{Rk}[Thm]{Remark}

\newcommand{\cat}[1]{\mathbf{#1}}
\newcommand{\op}{\mathrm{op}}

\newcommand{\Cat}{\cat{Cat}}
\newcommand{\Set}{\cat{Set}}

\newcommand{\Sub}{\cat{Sub}}

\newcommand{\msf}[1]{\mathsf{#1}}
	
\newcommand{\mbb}[1]{\mathbb{#1}}

\newcommand{\ov}[1]{\overline{#1}}

\newcommand{\dom}{\msf{dom}}

\newcommand{\To}{\Rightarrow}
\renewcommand{\to}{\rightarrow}

\newcommand{\yon}{\mathbf{y}}
\newcommand{\sheaf}{\mathbf{a}}
\newcommand{\Lan}{\msf{Lan}}
\newcommand{\PSh}{\msf{PSh}}
\newcommand{\Sh}{\msf{Sh}}
\newcommand{\incl}{\mathbf{i}}

\newcommand{\Par}{\msf{Par}}
\newcommand{\Total}{\msf{Total}}

\begin{document}
\leftmargini=2em
\title{Presheaves over a join restriction category}
\author{Daniel Lin}
\address{Department of Mathematics, Macquarie University, North Ryde, NSW 2109, Australia}
\email{daniel.lin@mq.edu.au}
\date{\today}
\thanks{The support of a Macquarie University Research Scholarship is gratefully acknowledged.}

\begin{abstract}
	Just as the presheaf category is the free cocompletion of any small category, there is an analogous
	notion of free cocompletion for any small restriction category. In this paper, we extend the work
	on restriction presheaves to presheaves over join restriction categories, and show that the
	join restriction category of join restriction presheaves is equivalent to some partial map category
	of sheaves. We then use this to show that the Yoneda embedding exhibits the category
	of join restriction presheaves as the free cocompletion of any small join restriction category.
\end{abstract}

\maketitle

\section{Introduction}\label{sec1}
The notion of a restriction category as a means of generalising partial map categories was formally introduced 
in \cite{CL1}, although the idea came much earlier from \cite{GRANDIS}.
In both papers, the partiality of a map was expressed in terms of an idempotent on its domain of definition, 
and this assignment of maps to their corresponding idempotents is known as the restriction structure on a 
category \cite{CL1}. It turns out that these restriction categories are the objects of a $2$-category called 
$\cat{rCat}$, and share many similarities with ordinary categories. Therefore, one might expect these restriction 
categories to have some notion of colimits and limits.

In \cite{GL}, the authors gave a definition of cocomplete restriction category, and introduced the notion
of a restriction presheaf over a restriction category. They also showed that there was an analogue of free cocompletion
in the restriction setting via the restriction category of restriction presheaves. The aim of this paper will be
to extend the work done in \cite{GL} to join restriction categories; namely restriction categories whose 
partially ordered hom-sets have joins if their elements are compatible.

We begin with a review of restriction categories \cite{CL1}, join restriction categories \cite{GUO} and $\M$-categories
in section \hyperref[sec2]{2}. In the same section, we characterise those $\M$-categories whose category
of partial maps has a join restriction structure, and define geometric $\M$-categories. In 
section \hyperref[sec3]{3}, we see that every geometric $\M$-category may be given a subcanonical topology,
and that the $\M$-category of sheaves on this site is also geometric. It turns out that the $\M$-category of sheaves
is the free cocompletion of any geometric $\M$-category. Using this fact, we give the free cocompletion
of any join restriction category in section \hyperref[sec4]{4}. Finally, in section \hyperref[sec5]{5}, we show that 
the partial map category of sheaves is equivalent to some join restriction category of join restriction presheaves. The
idea being that a compatible family of elements should correspond to some matching family for a covering sieve, and that 
the join of any such compatible family should correspond to an amalgamation. A consequence of this result is the analogue of 
free cocompletion for any join restriction category.

Finally, in this paper, unless otherwise stated, we shall assume that our categories are locally small.

\section{Join restriction categories and geometric \texorpdfstring{$\M$}{M}-categories}\label{sec2}
	Restriction categories were introduced in \cite{GRANDIS} as a way to represent the partiality of maps 
	by an idempotent on its domain of definition, known as a restriction idempotent. Many examples of restriction 
	categories exist and are listed in \cite{CL1}. For example, $\Set_p$ (the category of sets and partial functions)
	is a restriction category whose restriction idempotents are precisely the partial identity maps.
	In $\Set_p$, we may consider hom-sets of partial functions which agree on their domains of definition, and
	these form what is called a \emph{compatible family}. It is not hard to see that for any family of compatible
	maps $S$ in $\Set_p$, we may define a new partial function whose domain of definition is the \emph{union} of the
	domains of definition of the individual partial functions in $S$. Indeed, this idea is captured via the notion of a 
	join restriction category \cite{GRANDIS}. Examples of join restriction categories include $\Set_p$ as well
	as $\cat{Top}_p$, the category of topological spaces and partial continuous functions.
	
	In \cite{CL1}, the authors proved that the $2$-categories of split restriction categories and $\M$-categories were 
	$2$-equivalent. The goal of this section will be to prove that this $2$-equivalence restricts back
	to a $2$-equivalence between the $2$-category of split join restriction categories and the $2$-category of geometric 
	$\M$-categories. Let us begin by recalling the definition of a restriction category.

	\begin{Defn}[Cockett-Lack]
		A \emph{restriction category} is a category $\cat{X}$ together with assignations
			$$ \cat{X}(A,B) \to \cat{X}(A,A), \quad f \mapsto \bar{f} $$
		with $\bar{f}$ satisfying the following conditions:
			\begin{enumerate}[leftmargin=1.5cm,label=(R\arabic*)]
				\item $f \circ \bar{f} = f$;
				\item $\bar{g} \circ \bar{f} = \bar{f} \circ \bar{g}$;
				\item $\ov{g \circ \bar{f}} = \bar{g} \circ \bar{f}$;
				\item $\bar{h} \circ f = f \circ \ov{h \circ f}$,
			\end{enumerate}
		for suitable maps $g$ and $h$. We call $\bar{f}$ the \emph{restriction} of $f$.
	\end{Defn}

	A map $f \in \cat{X}$ is called a \emph{restriction idempotent} if $f = \bar{f}$, and is \emph{total} if $\bar{f} = 1$. 
	Observe that each hom-set $\cat{X}(A,B)$ has a partial order given by $f\le g$ if and only if $f = g\bar{f}$.
	If $\cat{X}$ and $\cat{Y}$ are restriction categories, then a \emph{restriction functor} $F\colon \cat{X}\to\cat{Y}$ is a 
	functor which preserves the restriction structure on $\cat{X}$. If $F,G \colon\cat{X}\to\cat{Y}$ are restriction functors, 
	a restriction transformation $\alpha\colon F\To G$ is a natural transformation whose components are total. Restriction categories, 
	restriction functors and restriction transformations form a $2$-category $\cat{rCat}$ \cite{CL1}.  
	
	In the restriction category $\Set_p$, the restriction idempotent on a partial function $f\colon A\rightharpoonup B$
	is given by the identity map on the domain of definition of $f$. Now suppose $g\colon A\rightharpoonup B$ is a partial function
	satisfying the condition $g\bar{f} = f\bar{g}$; in other words, $f$ and $g$ agree where their domains of definition intersect. 
	More generally, in any restriction category, there is a notion of such ``agreements'' between maps from the same hom-set.
	
	\begin{Defn}[Guo]
		Let $\cat{X}$ be a restriction category, and let $f,g \in \cat{X}(A,B)$. We say that $f$ and $g$ are \emph{compatible}
		if $f\bar{g} = g\bar{f}$, and denote this by $f \smile g$. For any set $S\subset \cat{X}(A,B)$, we say that
		$S$ itself is \emph{compatible} if maps in $S$ are pairwise compatible.
	\end{Defn}
	
	The following are a direct consequence of the definition of compatibility.
	
	\begin{Lemma}[Guo]\label{GuoCompatible}
		Let $\cat{X}$ be restriction category and suppose $f,g \in \cat{X}(A,B)$. Then
		\begin{enumerate}
			\item if $f\le g$, then $f \smile g$, and
			\item if $f\smile g$ and $\bar{f}=\bar{g}$, then $f=g$.
		\end{enumerate}
	\end{Lemma}

	\begin{proof}
		If $f\le g$, then $f=g\bar{f}$ which implies $f\bar{g}= (g\bar{f})\bar{g} = g\bar{g}\bar{f}=g\bar{f}$, or $f\smile g$.
		
		If $f\smile g$ and $\bar{f}=\bar{g}$, then $f= f\bar{f} = f\bar{g} = g\bar{f} = g\bar{g} = g$.
	\end{proof}

	We noted in $\Set_p$ that if two partial functions $f,g\colon A\rightharpoonup B$ satisfied the condition $g\bar{f} = f\bar{g}$, 
	then $f$ and $g$ agreed on the intersection of their domains of definition. From this, we can define a new partial function 
	called the \emph{join} of $f$ and $g$, $f\vee g \colon A \rightharpoonup B$, whose domain of definition is the union of the domains 
	of definition of $f$ and $g$. More generally, in any \emph{join restriction category}, if a family of maps from the same hom-set
	are compatible, then its join exists and satisfies the conditions below.
	
	\begin{Defn}[Guo]
		A join restriction category $\cat{X}$ is a restriction category such that for each $A,B \in\cat{X}$ and compatible set 
		$S\subset\cat{X}(A,B)$, the join $\bigvee_{s\in S} s$ exists with respect to the partial ordering on $\cat{X}(A,B)$, and
		furthermore, satisfies the following conditions:
		\begin{enumerate}[leftmargin=1.5cm,label=(J\arabic*)]
			\item $\ov{\bigvee_{s\in S} s} = \bigvee_{s\in S} \bar{s}$;
			\item $\left(\bigvee_{s\in S} s\right) \circ g = \bigvee_{s\in S} (s \circ g)$
		\end{enumerate}
		for suitable maps $f$ and $g$.
	\end{Defn}
	
	\begin{Prop}[Guo, Lemma 3.18]
		Let $\cat{X}$ be a join restriction category, and let $S\subset\cat{X}(A,B)$ be a compatible set. Then for any
		map $f \colon B\to C$,
			$$ f \circ \left(\bigvee_{s\in S} s\right) =  \bigvee_{s\in S} (f \circ s). $$
	\end{Prop}
	
	If $\cat{X}$ and $\cat{Y}$ are join restriction categories, then a join restriction functor $F\colon\cat{X}\to\cat{Y}$ is a functor
	which preserves the joins in $\cat{X}$. There is a $2$-category $\cat{jrCat}$ of join restriction categories, join restriction functors and
	restriction transformations. Note $\cat{jrCat}$ is a full sub-$2$-category of $\cat{rCat}$.

	Earlier in this section, we mentioned that one of the reasons for introducing restriction categories was to capture the
	notion of partiality of maps through idempotents. This suggests there is some relation between restriction categories and
	categories of partial maps. Indeed, this is the case, but only if we consider categories with a particular family of monomorphisms.

	\begin{Defn}[Cockett-Lack]
		An $\M$-category $(\cat{C},\M)$ consists of an underlying category $\cat{C}$, together with a
		family of monics $\M$ in $\cat{C}$ satisfying the following three conditions:
		\begin{enumerate}
			\item $\M$ contains all isomorphisms in $\cat{C}$;
			\item $\M$ is closed under composition; and
			\item If $m\colon C\to B$ is in $\M$ and $f\colon A\to B$ is any map, then the pullback of 
					$m$ along $f$ exists and is also in $\M$. This is called an $\M$-pullback.
		\end{enumerate}
	\end{Defn}
	A family of monics satisfying the above three conditions is called a \emph{stable system of monics} \cite{CL1}.
	If $m\colon C\to B$ is in $\M$, we refer to $m$ as an $\M$-subobject of $B$. More specifically, an $\M$-subobject
	of $C$ is an equivalence class of $\M$-maps $m\colon B\to C$, where $m \sim n$ if and only if
	there exists an isomorphism $\phi$ such that $m=n\phi$. There is a natural order on 
	these $\M$-subobjects; that is, for $m\colon B\to C$ and $n\colon D\to C$, we say $m\le n$
	if there is a (unique) map $f\colon B\to C$ such that $m=nf$.
	
	\begin{Defn}[Garner-Lin]
		An $\M$-category $(\cat{C},\M)$ is \emph{locally small} if $\cat{C}$ is locally small, and the $\M$-subobjects
		of any $C\in\cat{C}$ form a partially ordered set, which we denote by $\Sub_{\M}(C)$.
	\end{Defn}
	
	For the rest of this paper, our $\M$-categories will be assumed to be locally small unless otherwise stated.
	Clearly $\Sub_{\M}(C)$ has a greatest element, namely $1_C$, and has binary meets (and hence
	finite meets) given by pullback.
	
	If $(\cat{C},\M_{\cat{C}})$ and $(\cat{D},\M_{\cat{D}})$ are $\M$-categories, then a functor $F$ between them is called
	an \emph{$\M$-functor} if $Fm\in\M_{\cat{D}}$ whenever $m\in\M_{\cat{C}}$, and if $F$
	preserves $\M$-pullbacks. Also, given $\M$-functors $F$ and $G$, a natural transformation is called $\M$-cartesian
	if the naturality squares are pullback squares. These $\M$-categories, $\M$-functors
	and $\M$-cartesian natural tranformations form a $2$-category called $\M\Cat$ \cite{CL1}.
		
	Given any $\M$-category $(\cat{C},\M)$, we may form its category of partial maps $\Par(\cat{C},\M)$,
	whose objects are objects of $\cat{C}$, whose morphisms are spans $(m,f)$ (with $m\in\M$) and
	where composition is by pullback. It turns out that for any $\M$-category $(\cat{C},\M)$, its category of partial maps 
	$\Par(\cat{C},\M)$ has a restriction structure, where the restriction of any map $(m,f)$ is given by $(m,m)$. Further, 
	$\Par(\cat{C},\M)$ has the property that all of its restriction idempotents split. 
	
	There is a $2$-functor $\Par\colon\M\Cat\to\cat{rCat}_s$ from the $2$-category of $\M$-categories 
	to the sub-$2$-category of restriction categories whose idempotents split, and also a $2$-functor 
	$\M\Total\colon\cat{rCat}_s\to\M\Cat$; in fact, these $2$-functors are $2$-equivalences \cite{CL1}.

	Given that $\Par(\cat{C},\M)$ is a restriction category for any $\M$-category $(\cat{C},\M)$, it is natural to 
	ask what conditions $(\cat{C},\M)$ must satisfy for $\Par(\cat{C},\M)$ to be a join restriction category. 

	\begin{Defn}
		An $\M$-category $(\cat{C},\M)$ is called \emph{geometric} if $\Par(\cat{C},\M)$ is a join
		restriction category. An $\M$-functor $F\colon(\cat{C},\M_{\cat{C}})\to(\cat{D},\M_{\cat{D}})$
		between geometric $\M$-categories is also called \emph{geometric} if $\Par(F)$ is a join restriction functor.
	\end{Defn}

	There is a $2$-category $\cat{g}\M\Cat$ of geometric $\M$-categories, geometric $\M$-functors
	and $\M$-cartesian natural transformations. With the previous definition, the $2$-equivalence
	between $\cat{rCat}_s$ and $\M\Cat$ restricts back to a $2$-equivalence between the
	$2$-category of split join restriction categories $\cat{jrCat}_s$ and the $2$-category of geometric $\M$-categories:
	
	$$ \bfig
		\square/>`^{ (}->`^{ (}->`>/[\cat{jrCat}_s`\cat{g}\M\Cat`\cat{rCat}_s`\M\Cat.;\simeq```\simeq]
	\efig $$

	In \cite[Theorems 3.3.3, 3.3.5]{GUO}, the author gave a characterisation of these geometric $\M$-categories. However, 
	we will give a different characterisation using only elementary notions of pullbacks and colimits. In proving this
	theorem, we shall first define the \emph{matching diagram} for any family of $\M$-subobjects in $\cat{C}$,
	and also use a restatement of \cite[Lemma 1.6.20]{GUO}.
	
	\begin{Defn}
		Let $(\cat{C},\M)$ be an $\M$-category, and let $M = \{ m_i\colon A_i \to A\}_{i\in I}$ be a family
		of $\M$-subobjects of $A$. Denote the pullback of $m_i$ along $m_j$ as in the
		following diagram:
			$$ \bfig
				\square|alrb|[A_i A_j`A_i`A_j`A.;m_{ji}`m_{ij}`m_i`m_j]
				\efig $$
		We define the \emph{matching diagram} for $M$ as a diagram in $\cat{C}$ on the objects
		$\{ A_i \mid i\in I\} \cup \{ A_iA_j \mid i\ne j\}$, and with morphisms the family $\{ m_{ij} \mid i,j \in I \}$.
	\end{Defn}

	Observe that in any $\M$-category $(\cat{C},\M)$, any family of $\M$-subobjects in $\cat{C}$
	forms a cocone to its matching diagram. 
	
	\begin{Lemma}[Guo]\label{ParCMInequal}
		Suppose $(m,f),(n,g)\colon A\to B$ are two morphisms in the partial map category $\Par(\cat{C},\M)$,
		with $m\colon C\to A$ and $n\colon D\to A$. Then $(m,f)\le(n,g)$ if and only if there exists a (unique) arrow 
		$\phi\colon C\to D$ such that $n\phi=m$ and $g\phi=f$.
	\end{Lemma}

	\begin{Thm}\label{ParCMJoin}
		An $\M$-category $(\cat{C},\M)$ is geometric if and only if:
		\begin{enumerate}
			\item for any family of $\M$-subobjects $\{ m_i\colon A_i \to A\}_{i\in I}$, the colimit $\bigcup_{i\in I} A_i$ 
					of its matching diagram exists,
			\item the induced map $\bigvee_{i\in I}m_i\colon\bigcup_{i\in I} A_i\to A$ is in $\M$, and
			\item the colimit from (1), $\bigcup_{i\in I} A_i$, is stable under pullback.
		\end{enumerate}
	\end{Thm} 
	
	\begin{proof}
		We begin by proving the \emph{if} direction. Suppose $\{(m_i,f_i)\}_{i\in I}$ is a compatible
		family of maps from $A$ to $B$ in $\Par(\cat{C},\M)$, and so there is a unique map 
		$\mu=\bigvee_{i\in I}m_i\colon \bigcup_{i\in I} A_i\to A$ in $\M$. Now compatibility of $\{(m_i,f_i)\}_{i\in I}$ 
		means that the family $\{f_i\}_{i\in I}$ is a cocone to the matching diagram for $\{m_i\}_{i\in I}$
		\cite[Lemma 3.1.4]{GUO}. This induces a unique map $\gamma\colon \bigcup_{i\in I} A_i\to B$.
		$$ \bfig 
			\square|almm|<800,800>[A_i A_j`A_i`A_j`A;m_{ji}`m_{ij}`m_i`m_j]
			\morphism(0,0)<500,300>[A_j`\bigcup_{i\in I} A_i;a_j]
			\morphism(800,800)<-300,-500>[A_i`\bigcup_{i\in I} A_i;a_i]
			\morphism(500,300)|m|/-->/<300,-300>[\bigcup_{i\in I} A_i`A;\mu]
			\morphism(800,800)/{@{>}@/^15pt/}/<200,-1000>[A_i`B.;f_i]
			\morphism(0,0)|b|/{@{>}@/_10pt/}/<1000,-200>[A_j`B.;f_j]
			\morphism(500,300)|m|/{@{-->}@/^20pt/}/<500,-500>[\bigcup_{i\in I} A_i`B.;\gamma]
		\efig $$
		
		We claim that $(\mu,\gamma) = \bigvee_{i\in I} (m_i,f_i)$. To see this, first observe that
		$(m_i,f_i)\le(\mu,\gamma)$ for all $i\in I$ by applying Lemma~\ref{ParCMInequal}, since 
		$\mu a_i=m_i$ and $\gamma a_i=f_i$ by construction. Now suppose for each $i\in I$, we have
		$(m_i,f_i)\le(u,v)$, where $u\colon D\to A$ is a map in $\cat{C}$. This means that for each $i$, 
		there is a unique $\beta_i\colon C_i\to D$ such that $m_i=u\beta_i$ and $f_i=v\beta_i$. Since
		$m_i m_{ji} = m_j m_{ij}$ by construction, this implies that $\beta_i m_{ji} = \beta_j m_{ij}$
		(as $u$ is monic). In other words, the family $\{ \beta_i\}_{i\in I}$ is a cocone to the matching diagram for
		$\{m_i\}_{i\in I}$. Therefore, there exists a unique map $\delta\colon \bigcup_{i\in I} A_i\to D$
		such that $\beta_i = \delta a_i$, for all $i\in I$.
		
		Since $m_i = u\beta_i = (u\delta) a_i$ and $f_i = v\beta_i = (v\delta) a_i$, by uniqueness, we must
		have $\mu = u\delta$ and $\gamma = v\delta$ (as $\mu$ and $\gamma$ are the only maps
		satisfying the conditions $\mu a_i=m_i$ and $\gamma a_i =f_i$). Hence, $(\mu,\gamma)\le (u,v)$
		by Lemma~\ref{ParCMInequal}.
		
		To see that our definition of $(\mu,\gamma)$ satisfies (J1), note that by construction, 
		$(\mu,\mu) = \bigvee_{i\in I} (m_i,m_i)$, which means
		$$ \ov{\bigvee_{i\in I} (m_i,f_i)} = \ov{(\mu,\gamma)}= (\mu,\mu)=\bigvee_{i\in I} (m_i,m_i)
			= \bigvee_{i\in I} \ov{(m_i,f_i)}. $$
		It remains to show that $(\mu,\gamma)$ also satisfies (J2). So let $(x,y)\colon X\to A$ be a map
		in $\Par(\cat{C},\M)$. We need to show $\bigvee_{i\in I} [(m_i,f_i)(x,y)] = (\mu,\gamma)(x,y)$,
		or alternatively, $\bigvee_{i\in I} [(m_i,f_i)(1,y)] = (\mu,\gamma)(1,y)$ and
		$\bigvee_{i\in I} [(m_i,f_i)(x,1)] = (\mu,\gamma)(x,1)$ since $(x,y)=(1,y)(x,1)$.
		Now as composition in $\Par(\cat{C},\M)$ is the same as pulling back in $\cat{C}$, the 
		statement $\bigvee_{i\in I} [(m_i,f_i)(1,y)] = (\mu,\gamma)(1,y)$ is equivalent to 
		$y^*(\bigvee_{i\in I} m_i) = \bigvee_{i\in I} y^*(m_i)$, which is true as colimits are stable under 
		pullback by assumption. To show $\bigvee_{i\in I} [(m_i,f_i)(x,1)] = (\mu,\gamma)(x,1)$, simply
		note that the family $\{ x m_i \}_{i\in I}$ gives rise to the same matching diagram as for $\{m_i\}_{i\in I}$.
		
		In the \emph{only if} direction, let $\{m_i\colon A_i \to A\}_{i\in I}$ be a family of 
		$\M$-subobjects of $A$. As $\Par(\cat{C},\M)$ is a join restriction category, denote 
		$(\mu,\mu)=\bigvee_{i\in I} (m_i,m_i)$, where $\mu\colon \bigcup_{i\in I} A_i\to A$. Note that
		$\mu\in \M$ by definition. Also, since $(m_i,m_i)\le (\mu,\mu)$ for all $i\in I$, there exists 
		a unique $a_i\colon A_i\to \bigcup_{i\in I} A_i$ (for each $i\in I$) such that $m_i = \mu a_i$. Observe that 
		each $a_i\in \M$ as $a_i$ is a pullback of $m_i$ along $\mu$. We now show that the family 
		$\{a_i\}_{i\in I}$ is a colimit to the matching diagram for $\{m_i\}_{i\in I}$.
		
		Clearly $\{a_i\}_{i\in I}$ is a cocone to the matching diagram. Now let $\{b_i\colon A_i\to B\}_{i\in I}$ be 
		a cocone to the same matching diagram; that is, $b_i m_{ji} = b_j m_{ij}$ for each pair $i,j\in I$. But as this implies
		that the family $\{(m_i,b_i)\}_{i\in I}$ is compatible, we may take their join, which we
		denote by $(s,t) = \bigvee_{i\in I}(m_i,b_i)$. By join restriction axioms,
			$$ (s,s) = \ov{(s,t)} = \ov{\bigvee_{i\in I} (m_i,b_i)} 
					= \bigvee_{i\in I} (m_i,m_i) = (\mu,\mu), $$
		which means $s=\mu$ (up to isomorphism). But because $(m_i,b_i) \le (s,t) = (\mu,t)$,
		there exists an $\alpha_i$ such that $m_i = \mu \alpha_i$ and $b_i = t\alpha_i$ (for 
		every $i\in I$). However, $a_i$ is the only map such that $m_i = \mu a_i$. Therefore,
		we must have $\alpha_i = a_i$, which in turn implies that $b_i = ta_i$ for all $i\in I$. We
		need to show that $t$ is in fact the unique map with this property.
		
		So suppose $t'$ also satisfies the condition $b_i = t'a_i$. Then $(m_i,b_i)\le(\mu,t')$ for
		all $i\in I$, which means $(\mu,t) = \bigvee_{i\in I}(m_i,b_i) \le (mu,t')$. Therefore,
		$t=t'$ and so $\{a_i\}_{i\in I}$ is indeed the required colimit to the matching diagram for $\{m_i\}_{i\in I}$.
		
		Observe that by the previous argument, the family $\{m_i\}_{i\in I}$ will be a colimit
		to its matching diagram if and only if $\bigvee_{i\in I}(m_i,m_i) = (1,1)$ if and only if $\mu = 1$. With this 
		observation, it is easy to show that the colimit $\bigcup_{i\in I} A_i$ is stable under pullback by noting that pullbacks 
		in $\cat{C}$ are the same as composition in $\Par(\cat{C},\M)$ and applying join restriction axioms.		
	\end{proof}
	
	\begin{Rk}
		Substituting $I$ to be the empty set in 
		the above theorem tells us that if $(\cat{C},\M)$ is a geometric $\M$-category, then $\cat{C}$ must 
		have a strict initial object $0$, and that maps $0\to A$ are in $\M$ (for all $A\in\cat{C}$).
	\end{Rk}
	
	\begin{Ex}
		Every Grothendieck topos together with all monos is a geometric $\M$-category. This follows from 
		a generalisation of \cite[Proposition 1.4.3]{JOHNSTONE}. In particular, for every category $\cat{C}$ and 
		site $(\cat{C},J)$, the $\M$-categories $(\PSh(\cat{C}),\mbb{M})$ and $(\Sh(\cat{C}),\mbb{M})$ are geometric, 
		where $\mbb{M}$ denotes all monos in the respective categories.
	\end{Ex}
	
	The following result follows immediately from Theorem~\ref{ParCMJoin}.
	
	\begin{Prop}\label{GeomMFunctor}
		An $\M$-functor $F\colon(\cat{C},\M_{\cat{C}})\to(\cat{D},\M_{\cat{D}})$ between geometric $\M$-categories is 
		geometric if and only if $F$ preserves colimits of matching diagrams.
	\end{Prop}
	
	 By Theorem~\ref{ParCMJoin}, if $(\cat{C},\M)$ is a geometric $\M$-category, then any family of $\M$-subobjects 
	 $\{m_i\colon A_i\to A\}_{i\in I}$ has a join given by the induced map 
	 $\bigvee_{i\in I}m_i \colon \bigcup_{i\in I} A_i \to A$. In fact:

	\begin{Prop}
		If $(\cat{C},\M)$ is a geometric $\M$-category, then for all $C\in\cat{C}$, $\Sub_{\M}(C)$
		is a complete Heyting algebra and for each $f\colon D\to C$, the function 
		$f^*\colon \Sub_{\M}(C) \to \Sub_{\M}(D)$ preserves joins.
	\end{Prop}
	
	\begin{proof}
		Let $\{m_i\colon A_i\to C\}_{i\in I}$ be a family of $\M$-subobjects of $C$ and define the join of 
		the family of $\M$-subobjects to be the induced map $\bigcup_{i\in I} A_i \to C$.
		As the colimit $\bigcup_{i\in I} A_i$ is stable under pullback, it follows that 
		$f^*\colon \Sub_{\M}(C) \to \Sub_{\M}(D)$ preserves all joins. Furthermore, since $\Sub_{\M}(C)$ has all joins, it also 
		has all meets, and so it remains to show that joins distribute over meets.
		
		So let $m\colon B\to C$ be any $\M$-subobject. Then $m^*\colon\Sub_{\M}(C) \to \Sub_{\M}(B)$ has
		a left adjoint given by $m \circ (-) \colon \Sub_{\M}(B) \to \Sub_{\M}(C)$. But their composite is 
		$m^* \circ \Big( m \circ (-) \Big) = m \wedge (-)$, and both maps preserve joins. Therefore joins distribute over meets, 
		and $\Sub_{\M}(C)$ is a complete Heyting algebra.
	\end{proof}
	
\section{Free cocompletion of geometric \texorpdfstring{$\M$}{M}-categories}\label{sec3}
	
	In this section, we continue our discussion of geometric $\M$-categories. The goal of this section will be to show that
	every small geometric $\M$-category may be freely completed to a cocomplete geometric $\M$-category. Recall that
	an $\M$-category $(\cat{C},\M)$ is called \emph{cocomplete} if $\cat{C}$ is cocomplete and the inclusion
	$\cat{C} \hookrightarrow \Par(\cat{C},\M)$ preserves colimits \cite{GL}. The way we will show this is as follows. 
	
	First, we show that for every small geometric $\M$-category $(\cat{C},\M)$, its underlying category $\cat{C}$ may be given a 
	Grothendieck topology $J$. This allows us to form an $\M$-category of sheaves $(\Sh(\cat{C}),\M_{\Sh(\cat{C})})$ on this 
	site $(\cat{C},J)$, for some class of monics $\M_{\Sh(\cat{C})}$ in $\Sh(\cat{C})$. We then show that this $\M$-category 
	$(\Sh(\cat{C}),\M_{\Sh(\cat{C})})$ is the free cocompletion of any geometric $\M$-category $(\cat{C},\M)$.
	
	Let us begin with the following proposition.
	
	\begin{Prop}\label{Basis}
		Let $(\cat{C},\M)$ is a geometric $\M$-category and let $C\in\cat{C}$. Then there is a Grothendieck topology
		on $\cat{C}$ whose basic covers of $C\in\cat{C}$ are given by families of the following form:
			$$ \{ a_i \colon C_i\to C \mid a_i\in\M, \bigvee_{i\in I} a_i = 1 \text{ in } \Sub_{\M}(C) \}_{i\in I} $$
		Equivalently, by Theorem~\ref{ParCMJoin}, 
		$\{a_i\}_{i\in I}$ is a basic cover of $C$ if $C$ is the colimit of a matching diagram for some family of 
		$\M$-subobjects.
	\end{Prop}
	
	\begin{proof}
		Clearly $\{1_C\colon C\to C\}$ is a basic cover of $C$ as $1_C\in\M$. If $\{ a'_i\colon C_i \times_C D \to D\}_{i\in I}$
		is the pullback of $\{a_i\}_{i\in I}$ along $f\colon D\to C$, then $\{a'_i\}_{i\in I}$ is also a basic cover of $D$
		as colimits are stable under pullback.
		
		Finally, for each $i\in I$ and $C_i$, suppose $\{ b_{ij}\}_{j \in J_i}$ is a basic cover of $C_i$. We need
		to show that $\{a_i \circ b_{ij}\}_{i\in I, j\in J_i}$ is a cover of $C$. First note that 
		$a_i \circ b_{ij}\in\M$ for each $i\in I,j\in J_i$ as $\M$ is closed under composition. Then since
		\begin{align*}
			\bigvee_{i\in I,j\in J_i} a_i b_{ij} = \bigvee_{i\in I} a_i \circ \left( \bigvee_{j\in J_i} b_{ij} \right)
					= \bigvee_{i\in I}a_i = 1,
		\end{align*}
		the family $\{a_i \colon C_i\to C\}_{i\in I}$ above describes a basic cover of $C$ for each $C\in\cat{C}$.
	\end{proof}
	
	\begin{Lemma}
		Suppose $(\cat{C},\M)$ is a small geometric $\M$-category and let $J$ be the topology
		generated by the basis described in Proposition~\ref{Basis}. Then $J$ is subcanonical.
	\end{Lemma}
	
	\begin{proof}
		We need to show all representable presheaves on $\cat{C}$ are sheaves on the site $(\cat{C},J)$.
		So let $D\in\cat{C}$ and consider the representable $\yon D$. By Proposition 1 in
		\cite[p.123]{MM}, $\yon D$ is a sheaf if and only if for any basic cover 
		$R=\{a_i\colon C_i\to C\}_{i\in I}$, any matching family 
		$\{x_i \in (\yon D)(C_i)\}_{i\in I}$ for $R$ has a unique amalgamation. Consider the following
		pullback square (for some $i,j \in I$):
		$$ \bfig
			\square[C_i\times_C C_j`C_j`C_i`C.;m_{ij}`m_{ji}`a_j`a_i]
		\efig $$
		Let $\{x_i \in (\yon D)(C_i)\}_{i\in I}$ be a matching family for $R$. By definition, this implies
		that $x_i \circ m_{ji} = x_j \circ m_{ij}$, or that $\{x_i\}_{i\in I}$ is a cocone to the diagram
		for which $\{a_i\}_{i\in I}$ is a colimit. This means there exists a unique $x\colon C\to D$
		such that $x\circ a_i = x_i$ for all $i\in I$. In other words, this $x$ is the unique amalgamation
		of $\{x_i\}_{i\in I}$. Hence, the representable $\yon D$ is a sheaf, and $J$ is subcanonical.
	\end{proof}

	\subsection{Partial map category of sheaves}
	Recall from \cite{CL1} that if $(\cat{C},\M)$ is an $\M$-category, then there is an
	$\M$-category $\PSh_{\M}(\cat{C})$, or $(\PSh(\cat{C}),\M_{\PSh{(\cat{C})}})$, where 
	$\mu\colon P\To Q$ is in $\M_{\PSh(\cat{C})}$ if for every $\alpha\colon R\To Q$, there is an 
	$m \colon A\to B$ in $\M$ making the following a pullback square:
		$$ \bfig
			\square[\yon A`P`\yon B`Q.;`\yon m`\mu`\alpha]
		\efig $$
	
	We now consider the $\M$-category of sheaves on $\cat{C}$.
	
	\begin{Defn}
		Suppose $(\cat{C},\M)$ is a small geometric $\M$-category. Then there is an $\M$-category
		$\Sh_{\M}(\cat{C})$, or $(\Sh(\cat{C}),\M_{\Sh(\cat{C})})$, where $\mu\colon P\To Q$
		is in $\M_{\Sh(\cat{C})}$ if and only if for every map $\alpha\colon\sheaf\yon D\To Q$ in 
		$\Sh(\cat{C})$, there is a map $m\colon C\to D$ in $\M$ making the following a pullback square:
		$$ \bfig
			\square[\sheaf\yon C`P`\sheaf\yon D`Q,;`\sheaf\yon m`\mu`\alpha]
		\efig $$
		where $\sheaf\colon \PSh(\cat{C})\to\Sh(\cat{C})$ is the associated sheaf functor.
	\end{Defn}
	
	Note that $\M_{\Sh(\cat{C})} = \M_{\PSh(\cat{C})} \cap \Sh(\cat{C})$. 
	
	Recall the geometric $\M$-categories 
	$(\Sh(\cat{C}),\mbb{M})$ and $(\PSh(\cat{C}),\mbb{M})$, where 
	$\Sh(\cat{C})$ is the category of sheaves on some small site $(\cat{C},J)$ and $\mbb{M}$ represents all monics in 
	$\Sh(\cat{C})$ and $\PSh(\cat{C})$. Then the associated sheaf functor $\sheaf\colon \PSh(\cat{C})\to\Sh(\cat{C})$ is 
	also a geometric $\M$-functor from $(\PSh(\cat{C}),\mbb{M})$ to $(\Sh(\cat{C}),\mbb{M})$, as it not only 
	preserves all colimits, but also all finite limits.
	
	\begin{Lemma}\label{PreserveJoinM}
		Suppose $(\cat{C},\M)$ is a small geometric $\M$-category. Then the $\M$-functor 
		$\sheaf\yon\colon (\cat{C},\M) \to(\Sh(\cat{C}),\mbb{M})$ is geometric.
	\end{Lemma}
	
	\begin{proof}
		Let $\{ C_i \to D\}_{i\in I}$ be a family of $\M$-subobjects of $D$ in $\cat{C}$, and
		consider the basic cover $\{ C_i \to \bigcup_{i\in I} C_i \}_{i\in I}$. The associated
		covering sieve (as a subfunctor) is given by 
			\begin{equation}\label{CoveringSieve}
				\bigcup_{i\in I} \yon C_i \to \yon \left(\bigcup_{i\in I} C_i\right)
			\end{equation}
		in $\PSh(\cat{C})$, and the associated sheaf functor $\sheaf$ takes this map to an isomorphism in $\Sh(\cat{C})$. 
		Hence, as $\sheaf$ is a left adjoint, we have $\sheaf\yon \left( \bigcup_{i\in I} C_i\right) \cong \bigcup_{i\in I} \sheaf\yon C_i$.
	\end{proof}
	
	\begin{Thm}
		If $(\cat{C},\M)$ is a small geometric $\M$-category, then $\Sh_{\M}(\cat{C})$ is also geometric.
	\end{Thm}
	
	\begin{proof}
		First note that $\M_{\Sh(\cat{C})}$ is a subset of $\mbb{M}$ (all monics in $\Sh(\cat{C})$). So to show that 
		$(\Sh(\cat{C}),\M_{\Sh(\cat{C})})$ is geometric, if suffices to prove that if $\{ \alpha_i \colon P_i \to Q\}_{i\in I}$ 
		are $\M_{\Sh(\cat{C})}$-subobjects of $Q$, then the induced monic 
		$\bigvee_{i\in I}\alpha_i \colon \bigcup_{i\in I} P_i \to Q$ is also in $\M_{\Sh(\cat{C})}$. 
		
		Denote the pullback of each $\alpha_i$ along $\gamma\colon \sheaf\yon D\to Q$ by $\sheaf\yon m_i$, with
		$m_i\in \M$:
			$$ \bfig
				\square[\sheaf\yon C_i`P_i`\sheaf\yon D`Q.;`\sheaf\yon m_i`\alpha_i`\gamma]
			\efig $$
		Since $(\Sh(\cat{C}),\mbb{M})$ is geometric, by stability, the pullback of $\bigvee_{i\in I}\alpha_i$ along
		the same $\gamma$ is $\bigvee_{i\in I} \sheaf\yon m_i$:
			$$ \bfig
				\square[\bigcup_{i\in I} \sheaf\yon C_i`\bigcup_{i\in I}P_i`\sheaf\yon D`Q.;
						`\bigvee_{i\in I}\sheaf\yon m_i`\bigvee_{i\in I}\alpha_i`\gamma]
			\efig $$
		But observe that the following diagram commutes by Lemma~\ref{PreserveJoinM}:
			$$ \bfig
				\Vtriangle[\bigcup_{i\in I} \sheaf\yon C_i`\sheaf\yon \left(\bigcup_{i\in I}C_i\right)`\sheaf\yon D.;
					\cong`\bigvee_{i\in I}\sheaf\yon m_i`\sheaf\yon\left( \bigvee_{i\in I}m_i \right)]
			\efig $$
		Hence, the pullback of $\bigvee_{i\in I}\alpha_i$ along $\gamma$ is $\sheaf\yon\left( \bigvee_{i\in I}m_i \right)$, 
		and so $\Sh_{\M}(\cat{C})$ is geometric.
	\end{proof}

	We have shown that $(\Sh(\cat{C}),\M_{\Sh(\cat{C})})$ is geometric if $(\cat{C},\M)$ is small geometric. In the
	same way that $\Sh(\cat{C})$ is a cocomplete category, we would like to show 
	$(\Sh(\cat{C}),\M_{\Sh(\cat{C})})$ is cocomplete as an $\M$-category. Recall the following definition
	of a cocomplete $\M$-category and cocontinuous $\M$-functor.
	
	\begin{Defn}[Garner-Lin]
		An $\M$-category $(\cat{C},\M)$ is \emph{cocomplete} if the underlying category $\cat{C}$ is cocomplete,
		and the inclusion $\incl \colon \cat{C}\hookrightarrow\Par(\cat{C},\M)$ preserves colimits.
		An $\M$-functor between $\M$-categories is cocontinuous if it is cocontinuous as a functor
		between the underlying categories.
	\end{Defn}
	
	Observe that if $F\colon (\cat{C},\M_{\cat{C}}) \to (\cat{D},\M_{\cat{D}})$ is a cocontinuous
	$\M$-functor between geometric $\M$-categories, then it is also geometric. 
	
	\begin{Defn}
		A cocomplete geometric $\M$-category is one which is both cocomplete and geometric. This gives the 
		$2$-category $\cat{g}\M\cat{Cocomp}$ of cocomplete geometric $\M$-categories, cocontinuous $\M$-functors
		and $\M$-cartesian natural transformations.
	\end{Defn}
	
	If $(\cat{C},\M)$ is an $\M$-category with a terminal object $1$, there is a notion of a \emph{generic $\M$-subobject}, 
	or an $\M$-subobject classifier.
	
	\begin{Defn}
		Let $(\cat{C},\M)$ be an $\M$-category with a terminal object $1$. An $\M$-subobject classifier
		consists of an object $\Sigma \in\cat{C}$ and a map $\tau\colon 1\to\Sigma$ in $\M$, such that
		for any $m\in\M$, there exists a unique map $\tilde{m}$ making the following square a pullback:
			$$ \bfig
				\square[A`1`B`\Sigma.;`m`\tau`\tilde{m}]
			\efig $$
	\end{Defn}
	
	\begin{Prop}
		Let $(\cat{C},\M)$ be a small geometric $\M$-category. Then $(\Sh(\cat{C}),\M_{\Sh(\cat{C})})$ is
		cocomplete as an $\M$-category.
	\end{Prop}
	
	\begin{proof}
		Recall that as $\Sh(\cat{C})$ is locally cartesian closed, if $\Sh_{\M}(\cat{C})$ were to have an 
		$\M$-subobject classifier $\Sigma$, then the pullback functor 
		$\tau^*\colon \Sh(\cat{C})/\Sigma \to \Sh(\cat{C})$ will have a right adjoint. This will imply
		that $\Sh(\cat{C})$ has a partial map classifier for every $P \in \Sh(\cat{C})$ \cite{JOHNSTONE}, and in turn, 
		will mean that the inclusion $\Sh(\cat{C})\hookrightarrow\Sh_{\M}(\cat{C})$ has a right 
		adjoint \cite{CL2}. Therefore, it suffices to show that $\Sh_{\M}(\cat{C})$ has an
		$\M$-subobject classifier.
		
		From \cite[Example 3.7]{GL}, $\PSh_{\M}(\cat{C})$ has an $\M$-subobject classifier
		$\Sigma$ taking objects $C\in\cat{C}$ to $\Sub_{\M}(C)$, and morphisms $f$ to $f^*$
		(by pullback along $f$). There was also a map $\tau\colon 1\To\Sigma$ in $\M_{\PSh(\cat{C})}$. 
		We will show that this $\Sigma$ is a sheaf, and then because $\M_{\Sh(\cat{C})} = \M_{\PSh(\cat{C})} \cap \Sh(\cat{C})$,
		it follows that $\tau\colon 1\To\Sigma$ is an $\M$-subobject classifier in $\Sh_{\M}(\cat{C})$.
		
		We begin by showing that $\Sigma$ is a separated presheaf. Let $R = \{a_i\colon A_i \to C\}_{i\in I}$ be a 
		basic cover of $C\in\cat{C}$, and let $M = \{m_i \colon B_i \to A_i\}_{i\in I}$ be 
		a matching family for $R$. Now suppose $x,y \in\Sub_{\M}(C)$ are two amalgamations for $M$. Then 
		pulling either $x$ or $y$ back along $a_i$ gives $m_i$ (for all $i\in I$). That is,
		$ a_i^*(x) = a_i^*(y)$, and so post-composing both sides by $a_i$ yields
		$ a_i \wedge x = a_i \wedge y$.
		
		The families $\{a_i \wedge x\}_{i\in I}$ and $\{a_i \wedge y\}_{i\in I}$ are families of monics, so we 
		may take joins over all $i\in I$, giving 
		$ \bigvee_{i\in I} (a_i \wedge x) = \bigvee_{i\in I} (a_i \wedge y) $. However, since $\Sub_{\M}(C)$ is 
		a Heyting algebra, we get
			$$ x \wedge \bigvee_{i\in I} a_i = y \wedge \bigvee_{i\in I} a_i $$
		by distributivity, and so $x=y$ as $\bigvee_{i\in I} a_i = 1$ by definition. Therefore, $\Sigma$ is a separated
		presheaf. It remains to show that any matching family for $R$ has an amalgamation.
		
		Again, let $R$ be a basic cover of $C$ and $M$ a matching family for $R$ as above, and consider the following
		diagram:
			$$ \bfig
				\iiixiii/`>``>`>`>`>`>`>`>`>`>/[`A_jB_i`B_i`B_jA_i`A_iA_j`A_i`B_j`A_j`C.;
					`b_i``n_i`m_i`n_j`a_{ji}`b_j`a_{ij}`a_i`m_j`a_j]
				\efig $$
		The above squares are all pullback squares. To say that $m_i$ and $m_j$ belong to the same matching family is to say that 
		$n_i$ and $n_j$ represent the same $\M$-subobject, or equivalently $ a_{ji}^*(m_i) = a_{ij}^*(m_j)$ for
		all $i\in I$. Since $a_j a_{ij} = a_i a_{ji}$ by construction, post-composing both sides
		of the above equation by $a_j a_{ij} = a_i a_{ji}$ gives $ (a_i a_{ji}) a_{ji}^*(m_i) = (a_j a_{ij}) a_{ij}^*(m_j)$,
		or $a_i \circ (a_{ji} \wedge m_i) = a_j \circ (a_{ij} \wedge m_j)$ (using the fact 
		$a_{ij} \circ a_{ij}^*(m_j) = a_{ij} \wedge m_j$). But we can also rewrite the above equation as 
		$ a_im_ib_i = a_jm_jb_j $, or
			\begin{equation} \label{matchfam} a_j \wedge a_im_i = a_i \wedge a_j m_j \end{equation}
		by writing compositions as intersections (as all squares are pullbacks). So if $m_i, m_j$ come from a matching 
		family for $R$, then they must satisfy \eqref{matchfam}.
		
		We claim that the (unique) amalgamation for $M$ is $\bigvee_{i\in I} a_i m_i$. In other words, we need to show 
		that $a_j^*\left( \bigvee_{i\in I} a_im_i\right) = m_j$ for all $j\in I$. However, since $a_j$ is monic for all 
		$j\in I$, the previous equality holds if and only if 
		$a_j \circ a_j^*\left( \bigvee_{i\in I} a_im_i\right) = a_j \circ m_j$, which is true if and only if
			$$ a_j \wedge \left( \bigvee_{i\in I} a_im_i\right) = a_j \circ m_j, \quad \forall j\in I. $$
		Examining the left hand side, we have (by the distributive law),
			\begin{align*}
				a_j \wedge \left( \bigvee_{i\in I} a_im_i\right) &= \bigvee_{i\in I} (a_j \wedge a_im_i) 
					= (a_j \wedge a_jm_j) \vee \left(\bigvee_{i\ne j} a_j \wedge a_im_i \right) \\
					&= a_jm_j \vee \left(\bigvee_{i\ne j} a_i \wedge a_jm_j \right)
			\end{align*}
		using \eqref{matchfam} and the fact $a_j m_j \le a_j$ (as an $\M$-subobject). But for all $i\ne j$, 
		$a_i \wedge a_jm_j \le a_jm_j$ 
		(since there is an arrow $b_j$ from the domain of $a_i \wedge a_jm_j$ to $a_jm_j$). Therefore, this means 
		that $\left(\bigvee_{i\ne j} a_i \wedge a_jm_j \right) \le a_jm_j$. Hence,
			$$ a_j \wedge \left( \bigvee_{i\in I} a_im_i\right) = a_j m_j $$
		as required.
		
		So every matching family for $R$ has an amalgamation, implying that $\Sigma$ is indeed a sheaf. Therefore,
		when $(\cat{C},\M)$ is a geometric $\M$-category, $(\Sh(\cat{C}),\M_{\Sh(\cat{C})})$ is cocomplete as an 
		$\M$-category.
	\end{proof}

	We have established that if $(\cat{C},\M)$ is a geometric $\M$-category, then $(\Sh(\cat{C}),\M_{\Sh(\cat{C})})$ 
	is a geometric and cocomplete $\M$-category. We now show that $(\Sh(\cat{C}),\M_{\Sh(\cat{C})})$ is the
	free geometric cocompletion of any small geometric $\M$-category $(\cat{C},\M)$. The following lemmas will be useful.
	
	
	\begin{Lemma}[Kelly, Theorem 5.56] \label{FreeCocompLemma1}
		Let $F\colon \cat{C}\to \cat{D}$ be a functor, where $\cat{C}$ is a small site and $\cat{D}$ is cocomplete. Denote the left Kan 
		extension of $F$ along $\yon\colon\cat{C}\to\PSh(\cat{C})$ by $\tilde{F}\colon \PSh(\cat{C})\to\cat{D}$. Suppose the right
		adjoint to $\tilde{F}$ factors through the inclusion $\incl\colon\Sh(\cat{C})\to\PSh(\cat{C})$. Denote the category 
		of such functors $F\colon\cat{C}\to\cat{D}$ by $\Cat_i(\cat{C},\cat{D})$. Then the following is an equivalence of categories:
			$$ (-) \circ \sheaf\yon \colon \cat{Cocomp}(\Sh(\cat{C}),\cat{D}) \to \Cat_i(\cat{C},\cat{D}), $$
		with pseudo-inverse given by left Kan extension along $\sheaf\yon$.
	\end{Lemma}
	
	\begin{Lemma}\label{FreeCocompLemma2}
		Let $F\colon \cat{C}\to \cat{D}$ and $\tilde{F}\colon \PSh(\cat{C})\to\cat{D}$ be functors as above, and denote the right 
		adjoint to $\tilde{F}$ by $G\colon\cat{D}\to\PSh(\cat{C})$. Then for each $D\in\cat{D}$, $G(D)$ is a sheaf if and only if
		for all $C\in\cat{C}$, $\tilde{F}$ takes covering sieves $S \rightarrowtail \yon C$ in $\PSh(\cat{C})$ to isomorphisms
		in $\cat{D}$.
	\end{Lemma}
	
	\begin{proof}
		By definition, $S$ is a covering sieve if and only if for all $D\in\cat{D}$ and 
		$S \to G(D)$, there exists a unique extension $ \yon C \to G(D)$ making
		the following diagram commute:
			$$ \bfig
				\ptriangle/>` >->`<--/[S`G(D)`\yon C;
					`i`\exists!]
			\efig $$
		In other words, if and only if there is an isomorphism
			$$ \PSh(\cat{C})(i, G(D)) \colon
				\PSh(\cat{C})( S, G(D) ) \to \PSh(\cat{C})(\yon C,G(D)). $$
		As $\tilde{F} \dashv G$, the above is an isomorphism if and only if $\cat{D}(\tilde{F}i, D)$ is invertible for all $D\in\cat{D}$, 
		and this in turn is true if and only if $\tilde{F}i$ is invertible.
	\end{proof}
	
	\begin{Lemma}[Garner-Lin]
		If $F\colon (\cat{C},\M_{\cat{C}}) \to (\cat{D},\M_{\cat{D}})$ is an $\M$-functor with $(\cat{D},\M_{\cat{D}})$
		cocomplete, then $\tilde{F} = \Lan_{\yon} F$ is an $\M$-functor.
	\end{Lemma}

	\begin{Lemma}\label{FreeCocompLemma3}
		Let $(\cat{C},\M_{\cat{C}})$ and $(\cat{D},\M_{\cat{D}})$ be geometric $\M$-categories, with $(\cat{C},\M_{\cat{C}})$
		small and $(\cat{D},\M_{\cat{D}})$ cocomplete. Let $F\colon(\cat{C},\M_{\cat{C}})\to(\cat{D},\M_{\cat{D}})$ be an $\M$-functor. 
		Then $F$ preserves unions of $\M$-subobjects if and only if for all $C\in\cat{C}$, $\tilde{F}$ takes covering sieves 
		$S \rightarrowtail \yon C$ in $\PSh(\cat{C})$ to isomorphisms in $\cat{D}$.
	\end{Lemma}
	
	\begin{proof}
		By the previous lemma, $\tilde{F}$ is a cocontinuous $\M$-functor. This means that for any covering sieve
		$S\to\yon C$, and in particular, the covering sieve $\bigcup_{i\in I} \yon C_i \to \yon \left(\bigcup_{i\in I} C_i\right)$
		from \eqref{CoveringSieve}, we have
			$$ \tilde{F}\left( \bigcup_{i\in I} \yon C_i \right) \cong \bigcup_{i\in I} \tilde{F} \yon C_i \cong \bigcup_{i\in I} FC_i, $$
		and so $\tilde{F}\mu \colon \bigcup_{i\in I} FC_i \to F\left( \bigcup_{i\in I} C_i \right)$ is an isomorphism if and only if
		$F$ preserves unions of $\M$-subobjects.
	\end{proof}
	
	\begin{Thm}\label{FreeMCocomp}
		Let $(\cat{C},\M_{\cat{C}})$ and $(\cat{D},\M_{\cat{D}})$ be geometric $\M$-categories, with $(\cat{D},\M_{\cat{D}})$
		cocomplete. Then the following is an equivalence of categories:
			$$ (-) \circ \sheaf\yon\colon \cat{g}\M\cat{Cocomp}(\Sh_{\M}(\cat{C}), (\cat{D},\M_{\cat{D}})) \to 
				\cat{g}\M\Cat((\cat{C},\M_{\cat{C}}),(\cat{D},\M_{\cat{D}})) $$
	\end{Thm}
	
	\begin{proof}
		We first show that $(-) \circ \sheaf\yon$ is essentially surjective on objects. From \cite{GL}, we know that the following is an equivalence
		of categories:
			$$ (-) \circ \yon \colon \M\cat{Cocomp}(\PSh_{\M}(\cat{C}),(\cat{D},\M_{\cat{D}})) 
				\to \M\Cat((\cat{C},\M_{\cat{C}}),(\cat{D},\M_{\cat{D}})) $$
		So for every $F\colon(\cat{C},\M_{\cat{C}})\to(\cat{D},\M_{\cat{D}})$, there is a cocontinuous 
		$\tilde{F}\colon \PSh_{\M}(\cat{C})\to(\cat{D},\M_{\cat{D}})$. But applying Lemmas \ref{FreeCocompLemma3},
		\ref{FreeCocompLemma2} and \ref{FreeCocompLemma1} in succession gives a cocontinuous functor
		$\tilde{F}\incl \colon \Sh(\cat{C})\to\cat{D}$ such that $\tilde{F}\incl \sheaf\yon \cong F$, since $F$ is geometric. Now $\tilde{F}\circ\incl$ 
		is the composite of $\M$-functors, and so $(-) \circ \sheaf\yon$ is essentially surjective on objects.
		The fact $(-) \circ \sheaf\yon$ is fully faithful follows from Lemma \ref{FreeCocompLemma1}, and therefore $(-) \circ \sheaf\yon$ is an 
		equivalence of categories.
	\end{proof}

\section{Free cocompletion of join restriction categories}\label{sec4}
	In light of the $2$-equivalence between $\cat{g}\M\Cat$ and $\cat{jrCat}_s$, we may now use the previous result
	to give the free cocompletion of any join restriction category. Indeed, this is what we will do in this section.
	But let us begin with the definitions of a cocomplete restriction category, and cocomplete join restriction category.
	
	\begin{Defn}[Garner-Lin]
		A restriction category $\cat{X}$ is \emph{cocomplete} if it is split, its subcategory of total maps
		$\Total(\cat{X})$ is cocomplete, and the inclusion $\Total(\cat{X})\hookrightarrow\cat{X}$ 
		preserves colimits. Also, a restriction functor $F\colon\cat{X}\to\cat{Y}$ is called \emph{cocontinuous} if 
		the underlying functor $\Total(F)\colon\Total(X)\to\Total(Y)$ is cocontinuous. There is a
		$2$-category $\cat{rCocomp}$ of cocomplete restriction categories, cocontinuous restriction functors
		and restriction transformations.
	\end{Defn}

	\begin{Defn}
		A join restriction category $\cat{X}$ is \emph{cocomplete} if it is cocomplete as a restriction category.
		Also, a join restriction functor between join restriction categories $F\colon \cat{X}\to\cat{Y}$ is
		called \emph{cocontinuous} if $\Total(F)$ is cocontinuous. There is a $2$-category $\cat{jrCocomp}$ of
		cocomplete join restriction categories, cocontinuous restriction functors and restriction transformations.
	\end{Defn}
	
	Observe that we have omitted the term ``join'' in describing the $1$-cells of $\cat{jrCocomp}$. The reason
	for this is as follows. As $\Par$ and $\M\Total$ are $2$-equivalences, every split join restriction category $\cat{X}$
	may be rewritten as $\cat{X} \cong \Par(\Total(\cat{X}),\M_{\cat{X}})$, where $\M_{\cat{X}}$
	are the restriction monics in $\cat{X}$ \cite{CL1}. So if $F\colon \cat{X}\to\cat{Y}$ is a cocontinuous 
	restriction functor between split join restriction categories, then 
	$\M\Total(F)$ from $(\Total(\cat{X}),\M_{\cat{X}})$ to $(\Total(\cat{Y}),\M_{\cat{Y}})$ 
	is a cocontinuous $\M$-functor. But since cocontinuous $\M$-functors preserve joins of $\M$-subobjects,
	it follows that $F \cong \Par(\M\Total(F))$ is a join restriction functor by Proposition~\ref{GeomMFunctor}.
	
	We now describe the free cocompletion of any join restriction restriction category. Recall from \cite{CL1} that 
	the inclusion $\cat{r}\Cat_s \hookrightarrow \cat{r}\Cat$ has a left biadjoint $\msf{K}_r$, and the unit of this
	biadjoint $J$ at $\cat{X}$ is a restriction functor from $\cat{X}$ to $\msf{K}_r(\cat{X})$. It is easy to check that
	if $\cat{X}$ is a join restriction category, then so is $\msf{K}_r(\cat{X})$. Also, the fact
	$\cat{jrCocomp}$ and $\cat{g}\M\cat{Cocomp}$ are $2$-equivalent follows from their definitions. So 
	consider the following solid diagram:
	
	$$ \bfig
		\morphism(0,600)|a|/@{>}@<6pt>/<1000,0>[\cat{jr}\Cat`\cat{jr}\Cat_s;\msf{K}_r]
		\morphism(1000,600)|a|/@{{ (}->}@<6pt>/<-1000,0>[\cat{jr}\Cat_s`\cat{jr}\Cat;]
		\place(500,592)[\bot]
		\morphism(1000,600)|a|/@{>}@<6pt>/<1000,0>[\cat{jr}\Cat_s`\cat{g}\M\Cat;\M\Total]
		\morphism(2000,600)|b|/@{>}@<6pt>/<-1000,0>[\cat{g}\M\Cat`\cat{jr}\Cat_s;\Par]
		\place(1500,592)[\simeq]
		\morphism(2000,600)|r|/@{-->}@<6pt>/<0,-600>[\cat{g}\M\Cat`\cat{g}\M\cat{Cocomp}.;\Sh_{\M}]
		\morphism(2000,0)|l|/@{>}@<6pt>/<0,600>[\cat{g}\M\cat{Cocomp}.`\cat{g}\M\Cat;V]
		\morphism(1000,0)|a|/@{>}@<6pt>/<1000,0>[\cat{jrCocomp}`\cat{g}\M\cat{Cocomp}.;\M\Total]
		\morphism(2000,0)|b|/@{>}@<6pt>/<-1000,0>[\cat{g}\M\cat{Cocomp}.`\cat{jrCocomp};\Par]
		\place(1500,0)[\simeq]
		\morphism(1000,600)|r|/@{-->}@<6pt>/<0,-600>[\cat{jr}\Cat_s`\cat{jrCocomp};]
		\morphism(1000,0)|l|/@{>}@<6pt>/<0,600>[\cat{jrCocomp}`\cat{jr}\Cat_s;U]
	\efig $$
	
	Now let $\cat{X}$ be a small join restriction category. By Theorem \ref{FreeMCocomp}, the forgetful $2$-functor $V$ has a 
	left biadjoint at any \emph{small} geometric $\M$-category, as indicated by the dotted arrow above. It follows that $U$ also has a
	left biadjoint at any small join restriction category $\cat{X}$ given by $\Par(\Sh_{\M}(\M\Total(\cat{X})))$. Therefore,
	the following exhibits the codomain as the free join restriction cocompletion of $X$:
	\begin{equation}\label{Paray}
		 \eta_{\cat{X}}\colon \cat{X} \xrightarrow{J} \msf{K}_r(\cat{X}) \xrightarrow{\cong} \Par(\M\Total(\msf{K}_r(\cat{X}))) 
				\xrightarrow{\Par(\sheaf\yon)} \Par(\Sh_{\M}(\M\Total(\msf{K}_r(\cat{X})))),
	\end{equation}
	in the sense that the following is an equivalence of categories:
	\begin{equation}\label{Unit} 
		(-) \circ \eta_{\cat{X}} \colon\cat{jrCocomp}(\Par(\Sh_{\M}(\M\Total(\msf{K}_r(\cat{X})))),\E) \to \cat{jr}\Cat(\cat{X},\E).
	\end{equation}
	However, as we shall see in the next section, we may express the free cocompletion of any join restriction category in a 
	simpler form via the notion of join restriction presheaves.

\section{Equivalence between sheaves and join restriction presheaves}\label{sec5}
	We saw in the previous section that the free cocompletion of any join restriction category may be given by the partial map category of
	sheaves on some site. The aim of this section will be to present an equivalent category which is also the free cocompletion of
	any join restriction category. In order to do this, we need objects in this category to correspond with
	sheaves on $\Total(\msf{K}_r(\cat{X}))$. In particular, we need a corresponding notion of a matching family, and also that of amalgamation.
	
	\cite{GL} showed that the free cocompletion of a restriction category can be described in terms of restriction presheaves, the definition
	of which is given below (Definition~\ref{RestPresheaf}). As it turns out, the corresponding object we need is a presheaf over a join restriction 
	category (Definition~\ref{JoinRestPresheaf}), which is equipped with its own notion of compatibility and join. The corresponding notions of 
	matching families and amalgamation in a join restriction presheaf may then be described as follows.
	Instead of a matching family for a covering sieve, we have a compatible family of elements of the restriction
	presheaf, and instead of a unique amalgamation of such a matching family, we have a join of compatible families. These join
	restriction presheaves form a join restriction category, and we will show that this category is equivalent to some partial map
	category of sheaves, and hence show that it is indeed the free cocompletion of any join restriction category.
	
	Let us recall the definition of a restriction presheaf.
	
	\begin{Defn}[Garner-Lin]\label{RestPresheaf}
		Let $\cat{X}$ be a restriction category. A \emph{restriction presheaf} over $\cat{X}$ is a presheaf
		$P \colon \cat{X}^{\op}\to \Set$ equipped with a family of maps $\{F_A\}_{A\in\cat{X}}$
			$$ F_A\colon PA \to \cat{X}(A,A), \quad x \mapsto \bar{x} $$
		with each $\bar{x}$ being a restriction idempotent satisfying the following conditions:
		\begin{enumerate}[leftmargin=1.5cm,label=(RP\arabic*)]
			\item $x\cdot \bar{x} =x$;
			\item $\ov{x \cdot \bar{f}} = \bar{x} \circ \bar{f}$;
			\item $\bar{x} \circ g = g \circ \ov{x\cdot g}$.
		\end{enumerate}
	\end{Defn}
	
	For any restriction category $\cat{X}$, there is a restriction category called $\PSh_r(\cat{X})$, the objects of which are 
	restriction presheaves on $\cat{X}$, and the morphisms are natural transformations. In fact, $\PSh_r(\cat{X})$ is a 
	restriction category where the restriction on $\alpha\colon P\To Q$ is defined componentwise at $A\in\cat{X}$ by 
	$\alpha_A(x) = x\cdot\ov{\alpha_A(x)}$.
	
	Recall that any presheaf $P$ over an ordinary category $\cat{C}$ may be regarded as a profunctor from the terminal
	category $\cat{1}$ to $\cat{C}$, or as a bifunctor $P \colon \cat{C}^{\op} \times \cat{1} \to \Set$. Further recall that
	the collage of this $P\colon\cat{C}^{\op} \times \cat{1} \to \Set$, denoted here by $\tilde{P}$, is a category whose 
	objects are the disjoint union of the objects of $\cat{C}$ and $\star$, where $\star$ is the only object in $\cat{1}$
	\cite{Street}. Its hom-sets are defined as follows:
	\begin{align*}
		\tilde{P}(\star,\star) &= \cat{1}(\star,\star) = 1_{\star}; \\
		\tilde{P}(A,B) &= \cat{C}(A,B); \\
		\tilde{P}(A,\star) &= P(A,\star); \\
		\tilde{P}(\star,A) &= \emptyset.
	\end{align*}
	
	If $P\colon \cat{X}^{\op}\to\Set$ is a restriction presheaf, then its collage maybe given a canonical restriction structure.
	Conversely, if the collage of $P\colon\cat{X}^{\op}\to\Set$ is a restriction category, then $P$ may also be given a 
	restriction structure, making it a restriction presheaf. Therefore, the following two results from \cite{GL} follow automatically 
	from their analogues in \cite{CL1}.
	
	\begin{Lemma}
		If $\cat{X}$ is a restriction category, and $P$ is a restriction presheaf on $\cat{X}$, then for all $A\in\cat{X}, 
		x\in PA$ and maps $g\colon B\to A$, we have 
			$$\bar{g} \circ \ov{x\cdot g} = \ov{x\cdot g} $$
		and
			$$\ov{\bar{x} \circ g} = \ov{x\cdot g}. $$
	\end{Lemma}
	
	Note that for all $A\in\cat{X}$, the set $PA$ also has a partial ordering given by $x \le y$ if and only if
	$x = y \cdot \bar{x}$. As in the case of join restriction categories, we may define compatibility
	between elements of the same set $PA$.
	
	\begin{Defn}
		Let $\cat{X}$ be a restriction category and $P$ be a restriction presheaf over $\cat{X}$. For any $A\in\cat{X}$,
		we say that $x,y\in PA$ are compatible if $x\cdot \bar{y} = y\cdot \bar{x}$, and denote this by
			$x\smile y$.
	\end{Defn}
	
	We now present two analogous lemmas (see Lemma~\ref{GuoCompatible}) which follow from the definition of 
	compatibility.
	
	\begin{Lemma}\label{JRPcompatible}
		Let $\cat{X}$ be a restriction category and $P$ a restriction presheaf over $\cat{X}$. Let $A\in\cat{X}$ and
		$x,y\in PA$. Then
		\begin{enumerate}
			\item $x\le y$ implies $x \smile y$, and
			\item $x\smile y$ and $\bar{x}=\bar{y}$ implies $x=y$.
		\end{enumerate}
	\end{Lemma}
	
	The proof is essentially the same as for Lemma~\ref{GuoCompatible}. Again, it is the partial ordering which allows us 
	to define the notion of a join restriction presheaf over a join restriction category.

	\begin{Defn}\label{JoinRestPresheaf}
		Let $\cat{X}$ be a join restriction category. A \emph{join restriction presheaf} on $\cat{X}$ is a restriction
		presheaf $P\colon \cat{X}^{\op}\to\Set$ such that for all $A\in\cat{X}$ and all compatible subsets
		$S\subset PA$, the join $\bigvee_{s\in S} s$ exists with respect to the partial ordering on $PA$, and 
		satisfies the following conditions:
		\begin{enumerate}[leftmargin=1.5cm,label=(JRP\arabic*)]
			\item $\ov{\bigvee_{s\in S} s} = \bigvee_{s\in S} \bar{s}$;
			\item $\left( \bigvee_{s\in S} s \right) \cdot g = \bigvee_{s\in S} (s\cdot g)$
		\end{enumerate}
		for all $g\colon B\to A$ and $x\in PA$. Denote by $\PSh_{jr}(\cat{X})$, the full subcategory of $\PSh_r(\cat{X})$
		with join restriction presheaves as its objects.
	\end{Defn}
	
	Again, as in the case of restriction presheaves, it is not difficult to see that a presheaf $P\colon\cat{X}^{\op}\to\Set$
	is a join restriction presheaf if and only if its collage is a join restriction category. Therefore, as with the case
	for join restriction categories \cite[Lemma 3.1.8]{GUO}, the following proposition holds.
	
	\begin{Prop}
		Let $\cat{X}$ be a join restriction category, and let $P$ be a join restriction presheaf. Then for all
		$A\in \cat{X}$, $x\in PA$ and compatible $T\subset \cat{X}(B,A)$,
			$$ x\cdot \left(\bigvee_{t\in T} t \right) = \bigvee_{t\in T} (x\cdot t). $$
	\end{Prop}
	
	\subsection{Join restriction category of join restriction presheaves}
	We know that for any restriction category $\cat{X}$, the category of restriction presheaves on $\cat{X}$, $\PSh_r(\cat{X})$,
	is a restriction category. Now suppose $\cat{X}$ is a \emph{join} restriction category. By definition, $\PSh_{jr}(\cat{X})$
	is a restriction category. However, we will show that $\PSh_{jr}(\cat{X})$ is in fact also a join restriction category.
	
	\begin{Lemma}\label{JoinNatTrans}
		Let $\cat{X}$ be a join restriction category, and $P,Q$ join restriction presheaves on $\cat{X}$. Let $S$ be a 
		compatible set of pairwise natural transformations from $P$ to $Q$. Then the natural transformation 
		$\bigvee_{\alpha\in S} \alpha$ defined as follows:
			$$ \left(\bigvee_{\alpha\in S} \alpha\right)_A(x) = \bigvee_{\alpha\in S} \alpha_A(x) $$
		is the join of $S$, and furthermore, satisfies conditions (J1) and (J2).
	\end{Lemma}
	
	\begin{proof}
		We first have to show that $\bigvee_{\alpha\in S} \alpha$ is well-defined. That is, for all $\alpha,\beta\in S$,
		$\alpha_A(x) \ov{\beta_A(x)} = \beta_A(x) \ov{\alpha_A(x)}$ (for all $A\in\cat{X}$ and $x\in PA$). But this 
		follows by definition of restriction in $\PSh_r(\cat{X})$ and the naturality of $\alpha$ and $\beta$. It is also natural
		since for all $g\colon B\to A$,
			$$ \left[\left(\bigvee_{\alpha\in S} \alpha\right)_A(x)\right] \cdot g 
				= \left( \bigvee_{\alpha\in S} \alpha_A(x) \right) \cdot g 
				= \bigvee_{\alpha\in S} \left(\alpha_A(x) \cdot g \right)
				= \bigvee_{\alpha\in S} \alpha_A(x \cdot g) 
				= \left(\bigvee_{\alpha\in S} \alpha\right)_A(x\cdot g) $$
		using the fact $Q$ is a join restriction presheaf, and the naturality of $\alpha\in S$. To show that 
		$\bigvee_{\alpha\in S} \alpha$ really is the join, we have to show $\alpha' \le\bigvee_{\alpha\in S} \alpha$
		for all $\alpha'\in S$, or equivalently, $\alpha'_A(x) = 
		\left(\bigvee_{\alpha\in S} \alpha\right)_A\left(\ov{\alpha'}_A(x)\right)$. But this is true as
		\begin{align*}
			\left(\bigvee_{\alpha\in S} \alpha\right)_A\left(\ov{\alpha'}_A(x)\right)
				&= \left(\bigvee_{\alpha\in S} \alpha\right)_A\left(x\cdot \ov{\alpha'_A(x)}\right)
				= \left[\left(\bigvee_{\alpha\in S} \alpha\right)_A(x)\right] \cdot \ov{\alpha'_A(x)} \\
				&= \left[\bigvee_{\alpha\in S} \alpha_A(x) \right] \cdot \ov{\alpha'_A(x)}
				= \left(\alpha'_A(x) \cdot \ov{\alpha'_A(x)}\right) \vee 
					\bigvee_{\alpha\ne\alpha'} \alpha_A(x) \cdot \ov{\alpha'_A(x)} \\
				&= \alpha'_A(x) \vee \bigvee_{\alpha\ne\alpha'} \alpha'_A(x) \cdot \ov{\alpha_A(x)}
					= \alpha'_A(x)
		\end{align*}
		by compatibility and the fact $\alpha'_A(x) \cdot \ov{\alpha_A(x)} \le \alpha'_A(x)$. Also, if $\alpha\le\beta$
		for all $\alpha\in S$, then $\bigvee_{\alpha\in S} \alpha \le \beta$ since
		\begin{align*}
			\beta_A\left(\ov{\bigvee_{\alpha\in S} \alpha}_A(x)\right) 
				&= \beta_A\left(x\cdot \ov{\bigvee_{\alpha\in S} \alpha_A(x)}\right)
				= \beta_A(x) \cdot \bigvee_{\alpha\in S} \ov{\alpha_A(x)}
				= \bigvee_{\alpha\in S} \beta_A(x)\cdot \ov{\alpha_A(x)} \\
				& = \bigvee_{\alpha\in S} \alpha_A(x) =  \left(\bigvee_{\alpha\in S} \alpha\right)_A(x).
		\end{align*}
		Therefore, for any compatible set of natural transformations $S=\{ \alpha\colon P\To Q\}$,
		$\bigvee_{\alpha\in S} \alpha$ as defined previously is the join of $S$.
		
		To see that this join satisfies (J1), simply replace $\beta$ above by the identity. To see that (J2) is
		satisfied, let $\gamma\colon R\To P$ be a natural transformation and observe that
		$$ \left(\bigvee_{\alpha\in S} \alpha\right)_A( \gamma_A(x)) 
			=\bigvee_{\alpha\in S} \alpha_A(\gamma_A(x)) = \bigvee_{\alpha\in S} (\alpha\gamma)_A(x)
			= \left(\bigvee_{\alpha\in S} \alpha\gamma \right)_A(x). $$
		Therefore, the natural transformation $\bigvee_{\alpha\in S} \alpha$ defined as above really is the join of any 
		compatible $S \subset \PSh_r(\cat{X})(P,Q)$, and furthermore, satisfies conditions (J1) and (J2).
	\end{proof}
	
	The following proposition follows directly from Lemma~\ref{JoinNatTrans}.
	
	\begin{Prop}[Join restriction category of join restriction presheaves]
		Let $\cat{X}$ be a join restriction category. Then $\PSh_{jr}(\cat{X})$ is a join restriction category, with
		joins defined componentwise as in Lemma~\ref{JoinNatTrans} for any compatible
		subset $S \subset \PSh_{jr}(\cat{X})(P,Q)$.
	\end{Prop}
	
	The following are some properties of morphisms in $\PSh_{jr}(\cat{X})$.
		
	\begin{Prop}
		Let $\cat{X}$ be a join restriction category, and let $\alpha\colon P\To Q$ be a morphism in
		$\PSh_{jr}(\cat{X})$. Let $A\in\cat{X}$ and $S\subset PA$ be compatible. Then the set
		$\alpha_A(S) = \{ \alpha_A(x) \mid x\in PA\}$ is also compatible. In addition, if $x,y\in PA$
		with $x\le y$, then $\alpha_A(x)\le\alpha_A(y)$.
	\end{Prop}
	
	\begin{proof}\label{AlphaCompatible}
		Let $x,y\in S$, and observe that it is enough to show that 
		$\alpha_A(x)\cdot\ov{\alpha_A(y)} = \alpha_A(y\cdot \bar{x})$ (by interchanging $x$ and $y$
		and using the fact $x\smile y$). Since $\ov{\alpha_A(y)} \le \bar{y}$
		(as $\alpha_A(y) = \alpha_A(y\cdot \bar{y})$), we have
		\begin{align*}
			\alpha_A(x)\cdot\ov{\alpha_A(y)} &= \alpha_A\left(x \cdot\ov{\alpha_A(y)} \right)
				= \alpha_A\left(x \cdot \left(\bar{y} \circ \ov{\alpha_A(y)}\right) \right)
				= \alpha_A\left( \left(y \cdot \bar{x}\right) \cdot \ov{\alpha_A(y)} \right) \\
				&= \alpha_A\left( \bar{\alpha}_A(y) \right) \cdot \bar{x}
				= \alpha_A(y) \cdot \bar{x} = \alpha_A(y\cdot\bar{x}).
		\end{align*}
		Hence, $\alpha_A(S)$ is compatible if $S$ is compatible.
		
		Now if $x\le y$, then
		$$ \alpha_A(x) = \alpha_A(x)\cdot\ov{\alpha_A(x)} = \alpha_A(y\cdot\bar{x}) \cdot\ov{\alpha_A(x)}
			= \alpha_A(y)\cdot \left(\ov{\alpha_A(x)}\circ\bar{x}\right) 
			= \alpha_A(y)\cdot \ov{\alpha_A(x)} $$
		since $\ov{\alpha_A(x)} \le\bar{x}$. Therefore, $x\le y$ implies $\alpha_A(x)\le\alpha_A(y)$.
	\end{proof}
	
	\begin{Prop}
		Let $\alpha\colon P\To Q$ be a morphism in $\PSh_{jr}(\cat{X})$. Let $A\in\cat{X}$, and
		let $S\subset PA$ be compatible. Then
			$$ \alpha_A\left( \bigvee_{x\in S} x\right) = \bigvee_{x\in S} \alpha_A(x). $$
		In other words, components of natural transformations preserve joins.
	\end{Prop}
	
	\begin{proof}
		To prove equality, we will show they are compatible, and then show that their restrictions are equal.
		Now by definition, $x\le\bigvee_{x\in S} x$, which means 
		$\alpha_A(x) \le \alpha_A\left(\bigvee_{x\in S} x\right)$ by Proposition~\ref{AlphaCompatible}. Therefore, 
		$\bigvee_{x\in S} \alpha_A(x) \le \alpha_A\left(\bigvee_{x\in S} x\right)$, and hence 
		$\bigvee_{x\in S} \alpha_A(x) \smile \alpha_A\left(\bigvee_{x\in S} x\right)$.
		
		To show that their restrictions are equal, we first show 
		$\bar{\alpha}_A\left( \bigvee_{x\in S} x\right) = \bigvee_{x\in S} \bar{\alpha}_A(x)$. This is true since
		\begin{align*}
			\bar{\alpha}_A\left( \bigvee_{x\in S} x\right) 
				&= \left( \bigvee_{x\in S} x\right) \cdot \ov{\alpha_A\left( \bigvee_{y\in S} y\right)}
				= \bigvee_{x\in S} x \cdot \ov{\alpha_A\left( \bigvee_{y\in S} y\right)}
				= \bigvee_{x\in S} x \cdot \left(\ov{\alpha_A\left( \bigvee_{y\in S} y\right)} \circ\bar{x}\right) \\
				&= \bigvee_{x\in S} x \cdot \ov{\alpha_A\left( \bigvee_{y\in S} y\right) \cdot \bar{x}}
				= \bigvee_{x\in S} x \cdot \ov{\alpha_A\left( \bigvee_{y\in S} y \cdot \bar{x}\right)} \\
				&= \bigvee_{x\in S} x \cdot 
					\ov{\alpha_A\left( x\cdot\bar{x} \vee \bigvee_{y\ne x} y \cdot \bar{x}\right)}
				= \bigvee_{x\in S} x \cdot 
					\ov{\alpha_A\left( x \vee \bigvee_{y\ne x} x \cdot \bar{y}\right)} \\
				&= \bigvee_{x\in S} x \cdot \ov{\alpha_A(x)} 
				= \bigvee_{x\in S} \bar{\alpha}_A(x).
		\end{align*}
		Observing that $\ov{\alpha_A(x)} = \ov{\bar{\alpha}_A(x)}$, we then have
		\begin{align*}
			\ov{\alpha_A\left( \bigvee_{x\in S} x\right)} &= \ov{\ov{\alpha}_A\left( \bigvee_{x\in S} x\right)}
				= \ov{\bigvee_{x\in S} \bar{\alpha}_A(x)} = \ov{\bigvee_{x\in S} \alpha_A(x)},
		\end{align*}
		which means the restrictions of $\alpha_A\left( \bigvee_{x\in S} x\right)$ and 
		$\bigvee_{x\in S} \alpha_A(x)$ are equal. Therefore, as $\alpha_A\left( \bigvee_{x\in S} x\right)$
		and $\bigvee_{x\in S} \alpha_A(x)$ are compatible, they must be equal.
	\end{proof}
	
	Having introduced join restriction presheaves, our next goal is to show that for any small geometric $\M$-category 
	$(\cat{C},\M)$, $\Par(\Sh_{\M}(\cat{C}))$ and $\PSh_{jr}(\Par(\cat{C},\M))$ are equivalent as join restriction 
	categories. Recall that for any $\M$-category $(\cat{C},\M)$, there was an equivalence of $\M$-categories 
	$F\colon\PSh_{\M}(\cat{C}) \to \M\Total(\PSh_r(\Par(\cat{C},\M)))$, which on objects, takes presheaves
	$P$ on $\cat{C}$ to presheaves $\tilde{P}$ on $\Par(\cat{C},\M)$, with 
	$\tilde{P}(X)=\{ (m,f) \mid m\in\M, f\in P(\dom\,m)\}$ for all $X\in\Par(\cat{C},\M)$ \cite{GL}. By the fact
	that $\M\Cat$ and $\cat{rCat}_s$ are $2$-equivalent, we then have an equivalence of restriction categories
	$L\colon\Par(\PSh_{\M}(\cat{C})) \to \PSh_r(\Par(\cat{C},\M))$ (the transpose of $F$). Explicitly,
	$L=\Phi^{-1}_{\PSh_r(\Par(\cat{C},\M))} \circ \Par(F)$, where $\Phi_{\PSh_r(\Par(\cat{C},\M))}$ is the unit
	of the $2$-equivalence between $\M\Cat$ and $\cat{rCat}_s$. We will show that this equivalence $L$
	restricts back to an equivalence between join restriction categories $\Par(\Sh_{\M}(\cat{C}))$ and 
	$\PSh_{jr}(\Par(\cat{C},\M))$.
	
	However, let us first establish the following facts.

	\begin{Lemma}\label{ParShSub}
		Let $(\cat{C},\M)$ be a small geometric $\M$-category. Then $\Par(\Sh(\cat{C}),\M_{\Sh(\cat{C})})$ is a
		full subcategory of $\Par(\PSh(\cat{C}),\M_{\PSh(\cat{C})})$.
	\end{Lemma}
	
	\begin{proof}
		Let $P$ and $Q$ be sheaves on $\cat{C}$, and consider a morphism 
		$\incl P \xleftarrow{\mu} R \xrightarrow{\tau} \incl Q$ in $\Par(\PSh_{\M}(\cat{C}))$. We need to find
		a morphism $P \xleftarrow{\mu'} R' \xrightarrow{\tau'} Q$ in $\Par(\Sh_{\M}(\cat{C}))$ such that
		$(\incl \mu',\incl \tau')=(\mu,\tau)$. However, as this will be true if $\M_{\PSh(\cat{C})}$-subobjects
		of sheaves are sheaves, this is what we will prove.
		
		So let $\{ a_i\colon C_i\to C\}_{i\in I}$ be a basic cover of $C\in\cat{C}$, and let $R$ be an
		$\M_{\PSh(\cat{C})}$-subobject of $P$, where $P$ is a sheaf. Consider the subfunctor
		$\incl\colon S \rightarrowtail \yon C$, where $S$ is the covering sieve generated by our basic cover, and let 
		$\alpha\colon S \To R$ be any natural transformation. Since $P$ is a sheaf, there exists a unique extension 
		$\gamma\colon\yon C\To P$ making the outer square of the following diagram commute:
		
			$$ \bfig
				\pullback|mmmm|/>` >->` >->`-->/[\yon D`R`\yon C`P.;\beta`\yon m`\mu'`\gamma]/>`-->` >->/[S;\alpha``\incl]
			\efig $$
		Now pulling back $\mu' \colon R \rightarrowtail P$ along this unique extension $\gamma$ induces a unique map $S \To \yon D$. 
		But the fact $S$ is the covering sieve generated by our basic cover means that the following diagram commutes for every $i\in I$:
			$$ \bfig
				\Vtriangle/ >->`>`>/[C_i`D`C.;`a_i`m]
			\efig $$
		Since $\Sub_{\M}(C)$ is a complete Heyting algebra, taking the join of $\{a_i\}_{i\in I}$ means we have
		$ 1= \bigvee_{i\in I} a_i \le m$, or $m=1$. In other words, the map $\yon m$ is invertible, and so for every
		natural transformation, $S\To R$, there exists an extension $\yon C\To R$ given by the composite
		$\beta \circ (\yon m)^{-1}$. However, as $R$ is a subobject of a sheaf, and therefore separated, this implies that
		this extension is in fact unique. Hence, $R$ is a sheaf and 
		$\Par(\Sh_{\M}(\cat{C}))$ is a full subcategory of $\Par(\PSh_{\M}(\cat{C}))$.
	\end{proof}
	
	\begin{Thm}\label{Equivalence}
		Let $(\cat{C},\M)$ be a small geometric $\M$-category. Then $\Par(\Sh_{\M}(\cat{C}))$ and
		$\PSh_{jr}(\Par(\cat{C},\M))$ are equivalent as join restriction categories.
	\end{Thm}
	
	\begin{proof}
		Since $\Par(\Sh_{\M}(\cat{C}))$ is a full subcategory of $\Par(\PSh_{\M}(\cat{C}))$ for any geometric 
		$\M$-category $(\cat{C},\M)$ (Lemma~\ref{ParShSub}), let us consider the following solid diagram:
		$$ \bfig
			\square/-->`^{ (}->`^{ (}->`>/<1200,500>[\Par(\Sh_{\M}(\cat{C}))`\PSh_{jr}(\Par(\cat{C},\M))`\Par(\PSh_{\M}(\cat{C}))`\PSh_r(\Par(\cat{C},\M)).;
				L'```L]
		\efig $$
		We wish to show $L$ restricts to a functor $L'\colon\Par(\Sh_{\M}(\cat{C}))\to\PSh_{jr}(\Par(\cat{C},\M))$
		making the above diagram commute, and that $L'$ is an equivalence of join restriction categories.
		We will begin by showing that $L'$ is well-defined; that is, given a sheaf $P\colon \cat{C}^{\op} \to\Set$, we have to show 
		$\Par(F)(P)=F(P) = \tilde{P}\colon \Par(\cat{C},\M)^{\op}\to\Set$ is a join restriction presheaf.
		
		So let $X\in\Par(\cat{C},\M)$ and suppose $\{ (m_i,f_i)\}_{i\in I}$ is a compatible family of
		maps. That is, $f_i \cdot m_{ji} = f_j \cdot m_{ij}$ for any pair $i,j\in I$, where $m_{ji}$ is the
		pullback of $m_j$ along $m_i$. Since $\Par(\cat{C},\M)$ is a join restriction category by assumption,
		we may take the colimit $\{a_i\}_{i\in I}$ of the matching diagram for $\{ m_i\}_{i\in I}$. Let $\mu$ 
		be the induced map from this colimit.
		
		Now the condition $f_i \cdot m_{ji} = f_j \cdot m_{ij}$ for all $i,j\in I$ implies that $\{f_i\}_{i\in I}$
		is a matching family for the basic cover $\{a_i\}_{i\in I}$. But because $P$ is a sheaf, this implies
		the existence of a unique amalgamation $\gamma$ such that $\gamma \cdot m_i = f_i$ for all $i\in I$. 
		So define the join of $\{ (m_i,f_i)\}_{i\in I}$ to be $(\mu, \gamma \cdot \mu)$. It is then easy
		but tedious to check that the join restriction presheaf axioms hold, which means that $L'$ is 
		well-defined.
		
		Since $\Par(\Sh_{\M}(\cat{C}))$ and $\PSh_{jr}(\Par(\cat{C},\M))$ are both full subcategories,
		it also follows that $L'$ makes the above diagram commute. In addition, as $L$ is an equivalence
		of categories, this makes $L'$ fully faithful, and so it remains to show that $L'$ is essentially
		surjective on objects, and that $L'$ is a join restriction functor.
		
		To show $L'$ is essentially surjective, recall from \cite{GL} that there is an equivalence 
		$G\colon\PSh_r(\Par(\cat{C},\M))\to\Par(\PSh_{\M}(\cat{C}))$ of restriction categories, with
		$FG\cong 1$. On objects, $G$ maps restriction presheaves $P\colon\Par(\cat{C},\M)^{\op}\to\Set$
		to presheaves $\dot{P}\colon \cat{C}^{\op}\to\Set$, with $\dot{P}(X)=\{ x \mid x\in PX, \bar{x}=(1,1)\}$
		and $\dot{P}(f)=P(1,f)$. So if we can show that $G$ maps join restriction presheaves to 
		sheaves on $\cat{C}$, then $L'$ will be essentially surjective on objects.
		
		So let $P$ be a join restriction presheaf on $\Par(\cat{C},\M)$, and consider the presheaf
		$\dot{P}\colon \cat{C}^{\op}\to\Set$. Let $R=\{a_i\colon C_i\to C\}_{i\in I}$
		be a basic cover of $C$, and let $\{f_i \in \dot{P}(C_i)\}_{i\in I}$ be a matching family for $R$. 
		That is, $f_i \cdot \pi_j = f_j \cdot \pi_i$ for all $i,j\in I$, where $\pi_i,\pi_j$ are the pullbacks
		below:
		$$ \bfig
			\square[C_iC_j`C_i`C_j`C.;\pi_j`\pi_i`a_i`a_j]
		\efig $$ 
		Note that $\pi_i,\pi_j\in\M$. Now $\dot{P}$ will be a sheaf if we can find a unique $x\in\dot{P}(C)$ 
		such that $x\cdot a_i = f_i$. We will show that $x=\bigvee_{i\in I} f_i\cdot (a_i,1)$ is the unique
		amalgamation of $\{f_i\}_{i\in I}$. However, first we must show that such a join exists by
		showing $f_i\cdot (a_i,1) \smile f_j\cdot (a_j,1)$ for all $i,j\in I$. Now using the fact
		$f_i \cdot \pi_j=f_j\cdot \pi_i$ if and only if $f_i \cdot (1,\pi_j) = f_j \cdot (1,\pi_i)$ and
		$\ov{f_i} = (1,1)$ for all $i\in I$, we have
		\begin{align*}
			f_i\cdot (a_i,1) \cdot \ov{f_j\cdot (a_j,1)} &= f_i\cdot (a_i,1) \cdot \ov{\ov{f_j}\circ (a_j,1)}
				= f_i\cdot (a_i,1) \cdot (a_j,a_j) = f_i \cdot (a_j \pi_i,\pi_j) \\
				&= [f_i \cdot (1,\pi_j)]\cdot (a_j\pi_i,1) = [f_j \cdot (1,\pi_i)]\cdot (a_j\pi_i,1)
				= f_j \cdot (a_j \pi_i,\pi_i) \\
				&= f_j \cdot (a_i \pi_j,\pi_i) = f_j\cdot (a_j,1) \cdot \ov{f_i\cdot (a_i,1)}.
		\end{align*}
		So $x=\bigvee_{i\in I} f_i\cdot (a_i,1)$ exists. To see that it is in $\dot{P}(C)$, we have
			$$ \bar{x}= \ov{\bigvee_{i\in I} f_i\cdot (a_i,1)} = \bigvee_{i\in I} \ov{\ov{f_i}\circ (a_i,1)}
				= \bigvee_{i\in I} (a_i,a_i) = (1,1). $$
		
		We now check that $x$ is an amalgation of $\{f_i\}_{i\in I}$. That is, $x\cdot a_j = f_j$ for all
		$j\in I$, or equivalently, $x\cdot (1,a_j) = f_j$. Now
		\begin{align*}
			x\cdot(1,a_j) &= \left[\bigvee_{i\in I} f_i\cdot (a_i,1)  \right] \cdot(1,a_j) 
				= f_j \cdot(a_j,1)\cdot(1,a_j) \vee \bigvee_{i\in I-\{j\}} f_i\cdot (a_i,1) \cdot (1,a_j) \\
				&= f_j \vee \bigvee_{i\in I-\{j\}} f_i\cdot (\pi_i,\pi_j).
		\end{align*}
		But $\bigvee_{i\in I-\{j\}} f_i\cdot (\pi_i,\pi_j) \le f_j$ since
		\begin{align*}
			f_j \cdot \ov{\bigvee_{i\in I-\{j\}} f_i\cdot (\pi_i,\pi_j)} 
				&= f_j \cdot \bigvee_{i\in I-\{j\}} \ov{\ov{f_i}\circ (\pi_i,\pi_j)}
				= f_j \cdot \bigvee_{i\in I-\{j\}}  (\pi_i,\pi_i) \\
				&= f_j \cdot \bigvee_{i\in I-\{j\}}  (1,\pi_i) (\pi_i,1) 
				= \bigvee_{i\in I-\{j\}} [f_j \cdot (1,\pi_i)] \cdot (\pi_i,1) \\
				&= \bigvee_{i\in I-\{j\}} [f_i \cdot (1,\pi_j)] \cdot (\pi_i,1)
				= \bigvee_{i\in I-\{j\}} f_i \cdot (\pi_i,\pi_j).
		\end{align*}
		So $x\cdot (1,a_i) = x\cdot a_i = f_i$ for all $i\in I$, making $x=\bigvee_{i\in I} f_i\cdot (a_i,1)$
		an amalgation of $\{f_i\}_{i\in I}$. It remains to show that such an $x$ is unique.

		So let $y\in\dot{P}(C)$ satisfy the condition $y\cdot a_i = y\cdot (1,a_i) = f_i$ for all $i\in I$.
		Then $y\cdot (1,a_i) = x\cdot (1,a_i)$ implies 
		$\bigvee_{i\in I}y\cdot (1,a_i)(a_i,1) = \bigvee_{i\in I}x\cdot (1,a_i)(a_i,1)$, which in turn implies
		$x=y$ since $\bigvee_{i\in I}(a_i,a_i)=(1,1)$. Therefore, if $P$ is a join restriction presheaf
		on $\Par(\cat{C},\M)$, then $\dot{P}$ is a sheaf, and so $L'$ is essentially surjective on objects.
		
		Finally, to show that $L'$ is a join restriction functor, note that $L'$ is a restriction functor as $L$ is a 
		restriction functor. Furthermore, $\Total(L')\colon \Sh(\cat{C})\to\Total(\PSh_{jr}(\Par(\cat{C},\M)))$ 
		is an equivalence of categories, with pseudo-inverse given by $\Total(G)$ restricted back to 
		$\Total(\PSh_{jr}(\Par(\cat{C},\M)))$. Therefore, as $\Total(L')$ is 
		cocontinuous, $L'$ must be a join restriction functor, and 
		hence $\Par(\Sh_{\M}(\cat{C}))$ and $\PSh_{jr}(\Par(\cat{C},\M))$ are equivalent as join restriction categories.
	\end{proof}
	
	\begin{Cor}
		For any small join restriction category $\cat{X}$, the Yoneda embedding 
		$\yon_{jr} \colon \cat{X}^{\op} \to \PSh_{jr}(\cat{X})$ exhibits the category of join restriction presheaves 
		$\PSh_{jr}(\cat{X})$ as its free cocompletion, in the sense that the functor
			$$ (-) \circ \yon_{jr} \colon \cat{jrCocomp}(\PSh_{jr}(\cat{X}),\E) \to \cat{jrCat}(\cat{X},\E) $$
		is an equivalence of categories for any cocomplete join restriction category $\E$.
	\end{Cor}

	\begin{proof}
		The composite
		$$ \cat{X} \xrightarrow{J} \msf{K}_r(\cat{X}) \xrightarrow{\cong} \Par(\M\Total(\msf{K}_r(\cat{X}))) 
				\xrightarrow{\Par(\sheaf\yon)} \Par(\Sh_{\M}(\M\Total(\msf{K}_r(\cat{X})))) 
				\xrightarrow{\simeq} \PSh_{jr}(\cat{X}) $$
		from \eqref{Paray} is naturally isomorphic to $\yon_{jr}$ by the same argument as
		presented in \cite[Theorem 4.12]{GL}. Since precomposition with \eqref{Unit} is an equivalence of 
		categories, it follows that precomposition with $\yon_{jr}$ is also an equivalence of categories.
	\end{proof}


\bibliographystyle{plain}
\bibliography{Presheaves_over_a_join_restriction_category_-_Lin}

\end{document}